\documentclass{amsart}
\usepackage{amssymb}
\usepackage[english]{babel}
\usepackage{amsthm}

 \newtheorem{thm}{Theorem}[section]
 \newtheorem{cor}[thm]{Corolary}
 \newtheorem{lem}[thm]{Lemma}

 \theoremstyle{definition}
 \newtheorem{exa}[thm]{Ejemplo}

\newcommand{\C}{{\mathbb C}}
\newcommand{\N}{{\mathbb N}}

\newcommand{\B}{{\mathbb B}}
\newcommand{\SB}{{\mathbb S}}

\begin{document}

\title{Self-adjoint Representations of Braid Groups}

\author{Claudia Mar\'{i}a Egea, Esther Galina}

\address{Facultad de Matem\'atica Astronom\'{i}a y F\'{i}sica,
Universidad Nacional de C\'ordoba, C\'ordoba, Argentina}
\email{ cegea@mate.uncor.edu, galina@mate.uncor.edu}
\begin{abstract}
We give a method to construct new self-adjoint representations of the braid group. In particular, we give a family of irreducible self-adjoint representations of dimension arbitrarily large. Moreover we give sufficient conditions for a representation to be constructed with this method.
\end{abstract}

\keywords{Braid Group; Irreducible Representations.}

\subjclass{ 20C99, 20F36}
\thanks{This work was partially supported by CONICET, SECYT-UNC, FONCYT.}
\maketitle
\section{Introduction}\label{Introduccion}

The braid group of $n$ string, $\B_n$, is presented by $n-1$ generators, $\tau_{1}, \dots,\tau_{n-1}$, and the following relations
$$
\left\{\begin{array}{ll}
       \tau_k\tau_j=\tau_j\tau_k, &\textrm{ whenever  }|k- j|>1;  \\
        \tau_k\tau_{k+1}\tau_k=\tau_{k+1}\tau_{k}\tau_{k+1},  &\ 1\leq k\leq n-2;
        \end{array}\right.
$$

The objective of this work is to give a step forward in the classification of the irreducible representations of $\B_n$.  As long as we known,
there are only few contributions in this sense, some known results are the following ones. The classification is completed for representations of dimension lower or equal to $n$. Formanek classified all the irreducible representations
of $\B_n$ of dimension lower than $n$ \cite{F}. Sysoeva did it for dimension equal to $n$ \cite{S}. Larsen and Rowell gave some results for unitary
representation of $\B_n$ of dimension multiples of $n$. In particular, they proved there are not irreducible representations of dimension $n+1$ \cite{LR}.
Levaillant proved when the Lawrence-Krammer representation is irreducible and when it is reducible \cite{L}.

In a previous work \cite{EG1}, we gave a family of irreducible representations of $\B_n$ described as follows. Let $m< n$ and let X be the set of $n$-tuples with $m$ ones and $n-m$ zeros. For example, if $n=4$ and $m=2$, $$X=\{(0,0,1,1),(0,1,0,1),(0,1,1,0),(1,0,0,1), (1,0,1,0),(1,1,0,0)\}$$

Let $V$ be a complex vector space with orthonormal basis $\beta=\{v_x: x\in X\}$. Note that the permutation group $\SB_n$ acts on $X$ permuting the coordinates. For each $x\in X$ choose a complex number $q_{x_k, x_{k+1}}$, which depends only on components $k$ and $k+1$ of $x$. We defined a representation $\phi_m:\B_n \rightarrow \textrm{GL} (V)$, given by
$$
\phi_m(\tau_k)(v_x)= q_{x_k, x_{k+1}} v_{\sigma_k(x)}
$$
where $\sigma_k$ is the transposition which permutes the places $k$ and $k+1$, and

$$q_{x_k, x_{k+1}}=\left
\{\begin{array}{ll}
    1       &\textrm{    if   } x_k=x_{k+1}   \\
    t       &\textrm{    if   } x_k\neq x_{k+1}
\end{array}\right.
$$ with $t$ a real number and $t\neq 0, 1, -1$.
If $n=4$ and $m=2$, the matrix are
$$
\phi_2(\tau_1)= \left(\begin{array}{cccccccccc}
          \  1      &\  \   &\  \      &\  \   &\  \        &\  \   \\
          \  \      &\  0   &\  0      &\  t   &\  0        &\  \    \\
          \ \       &\  0   &\  0      &\  0   &\  t        &\  \    \\
          \ \       &\  t   &\  0      &\  0   &\  0        &\  \    \\
          \ \       &\  0   &\  t      &\  0   &\  0        &\  \    \\
          \ \       &\  \   &\  \      &\  \   &\  \        &\  1
\end{array}\right),
\phi_2(\tau_2)= \left(\begin{array}{cccccccccc}
          \  0      &\  t   &\  \      &\  \   &\  \   &\  \     \\
          \  t      &\  0   &\  \      &\  \   &\  \   &\  \     \\
          \  \      &\  \   &\  1      &\  \   &\  \   &\  \     \\
          \  \      &\  \   &\  \      &\  1   &\  \   &\  \     \\
          \  \      &\  \   &\  \      &\  \   &\  0   &\  t     \\
          \  \      &\  \   &\  \      &\  \   &\  t   &\  0

\end{array}\right)
$$

$$
\phi_2(\tau_3)= \left(\begin{array}{cccccccccc}
          \ 1      &\ \   &\ \      &\ \   &\ \   &\ \     \\
          \ \      &\ 0   &\ t      &\ \   &\ \   &\  \    \\
          \  \     &\ t   &\ 0      &\ \   &\ \   &\ \     \\
          \  \     &\ \   &\ \      &\ 0   &\ t   &\ \     \\
          \  \     &\ \   &\ \      &\ t   &\ 0   &\ \     \\
          \  \     &\  \  &\ \      &\ \   &\ \    &\ 1
\end{array}\right)
$$

Let $n>2$, then $(\phi_m, V_m)$ is an irreducible representation of $\B_n$ of dimension
$\frac{n!}{m!(n-m)!}$, for all $1\leq m<n$. In particular, if $m=1$ is equivalent to the standard representation which is the unique irreducible representation of $\B_n$ of dimension $n$ \cite{S}. If $m=2$ this representations has dimension $n(n+1)/2$ as well as the Lawrence-Krammer representation, but they are not equivalent. Indeed, $\phi_m$ satisfies the condition (\ref{condicion (2)}) that is formulates below and the Lawrence-Krammer representation does not. Therefore, we are almost certainly that these representations are not equivalent to any other known representation.

Note that these representations satisfy the following properties for all $j,k$, $1\leq j,k\leq n-1$,
\begin{equation} \label{condicion (1)}
\phi_m(\tau_k) \textrm{    is a self-adjoint operator and }  \phi_m(\tau_k)^2\neq \lambda 1_V, \textrm{       for all } \lambda\in \C,
\end{equation}

\begin{equation}
\label{condicion (2)} [\phi_m(\tau_k)^2, \phi_m(\tau_j)^2]=0,
\end{equation}
On the other hand, the spectral decomposition of $\phi_m(\tau_k)^2$ is $t^2 P_{k,0} +  P_{k,1}$ for the appropriated projections $P_{k,l}$ and
$$
\phi_m(\tau_k) P_{j,l} \phi_m(\tau_k)^{-1}=\left
\{\begin{array}{ll}
     P_{j,l}                                  &\textrm{    if   }j\neq k+1 \\
    P_{k,0}P_{k+1,0}+P_{k,1}P_{k+1, 1}       &\textrm{    if   }j=k+1, \ \  l=0   \\
    P_{k,1}P_{k+1,0}+P_{k,0}P_{k+1, 1}       &\textrm{    if   }j=k+1, \ \  l=1
\end{array}\right.
$$

That is,
\begin{equation}\label{condicion (3)}
 \phi_m(\tau_k) P_{j,l} \phi_m(\tau_k)^{-1}\in N  \textrm{  for all spectral projections }  P_{j,l} \textrm{  of   } \phi_m(\tau_j)^2,
\end{equation}
where $N$ is the algebra of linear operators generated by the spectral projections
$\{P_{j,l}: 1\leq j\leq n-1, \  l=0,1\}$.

In this work we are going to parametrize, through $4$-tuples $(X, \pi, \nu, U)$, all the representations of $\B_n$ with these
three properties (see Theorem \ref{teorema 2}). In particular, they are representations of the quotient group of $\B_n$ given by the relation (\ref{condicion (2)}). Note that this quotient is a covering of $\SB_n$.

Moreover, we can give sufficient conditions for these representations to be irreducible (see Theorem \ref{irreduc}). As an important consequence of this parametrization, we construct explicit examples of irreducible representations of dimension arbitrarily large. More precisely, Theorem \ref{dimgrande} says that for  any positive integer $M$, there exists an irreducible representation of $\B_n$ of dimension greater than $M$.

The family of representations that satisfy the conditions (\ref{condicion (1)}), (\ref{condicion (2)}) and
 (\ref{condicion (3)}) contains a subfamily of some known ones as the diagonal local representations (see section \ref{Ejemplos}).

This work is divided in $4$ sections. In the section $2$, we give a method to obtain non unitary self-adjoint representations of $\B_n$ from a $4$-tuple $(X, \pi, \nu, U)$ and we parametrize the representations of the braid group mentioned before through of $4$-tuples. In section $3$, we
present the corresponding parameters of some known representations as certain diagonal local representations. Finally, in section $4$ we give some conditions in the $4$-tuple to guarantee the irreducibility of the associated representation and we present a family of irreducible representations of dimension arbitrarily large.


\section{Parametrization of non unitary self-adjoint representations of $\B_{n}$} \label{teoFundamental}

In this section we are going to construct representations of $\B_n$ and we are going to parametrize those that satisfy three special conditions. We start recalling the definition of {\it{cocycle}}.

Let $G$ be a discrete group, $\pi$ be an action of $G$ on $X$, and let $\nu:X\rightarrow \N$ be a $\pi$-invariant function, that is for each $g\in G$, $\nu(\pi(g)x)=\nu(x)$ for all $x\in X$. For each non negative integer $m$ we fix a vector space $V_m$ of dimension $m$ and for each $x\in X$, if $\nu(x)=m$, we choose the vector space $V_x=V_m$. Let $V=\bigoplus_{x\in X} V_x$. A cocycle is a function $U:G\times X\rightarrow \textrm{GL}(V)$ which satisfies the following properties:
\begin{enumerate}
\item  for each $x\in X$, $U(\cdot, x): G\rightarrow \textrm{ GL}(V_x)$,
\item  $U(1,x)=1_{V_x} $, for all $x\in X$,
\item  $U(g_1g_2,x)=U(g_1,\pi(g_2)x)U(g_2,x)$ for all $x\in X$ and for all $g_1, g_2\in G$.
\end{enumerate}

If $G$ is a discrete group presented by generators and relations, it is enough to define a cocycle $U$ in the generators with some relations that come from the relations of $G$. In the case $G=\B_{n}$, the equations of cocycle for the generators are,

\begin{equation} \label{ecuacion 1}
\begin{array}{rl}
&U(\tau_k, \pi(\tau_{k+1})\pi(\tau_k) x) U(\tau_{k+1}, \pi(\tau_k) x) U(\tau_k, x)  = \\
   & \qquad\qquad\qquad \qquad \quad = U(\tau_{k+1},\pi(\tau_k)\pi( \tau_{k+1})x) U(\tau_k,\pi(\tau_{k+1})x) U(\tau_{k+1}, x)
   \end{array}
\end{equation}
if $1\leq k\leq n-2$
\begin{equation}\label{ecuacion 2}
U(\tau_j,\pi(\tau_{k})x) U(\tau_k, x) =  U(\tau_k, \pi(\tau_j)x) U(\tau_j, x)
\end{equation}
 if $|j-k|>1$.

From now on fix the notation for $\pi_k x=\pi(\tau_k)x$. We can define the following representations of $\B_n$.
Let $(X, \pi, \nu, U)$ be a $4$-tuple where
\begin{enumerate}
\item $X$ is a finite set;
\item $\pi$ is an action of $\B_n$ on $X$;
\item $\nu:X\rightarrow \N$ is a $\pi$-invariant function, that is for each $k$, $\nu(\pi_k(x))=\nu(x)$ for all $x\in X$;
\item $U: \B_n \times X\rightarrow \textrm{GL}(V)$ is a cocycle, where $V:=\bigoplus_{x\in X} V_x$ and $V_x:=V_m$ if $\nu(x)=m$. Note that each element in $V$ can be written by $\sum_{x\in X} x\otimes v_x$.
\end{enumerate}
Then
$$
\phi=\phi_{(X, \pi, \nu, U)}: \B_n \rightarrow \textrm{GL(V)}
$$ defined by
\begin{equation} \label{representacion}
\phi(\tau_k)( \sum_{x\in X} x\otimes v_x )=\sum_{x\in X} x\otimes U(\tau_k,\pi_k^{-1} x) v_{\pi_k^{-1} x}
\end{equation}
is a representation of $\B_n$ on $V$.

Observe that if the action is transitive, that is $X$ is the orbit of an element $x_0\in X$, then $V$ is the set of global sections of a trivial vector bundle, because $V_x=W$ for all $x\in X$ and $V=X\times W$.

\begin{exa}
Let $\rho$ be a representation of $\B_{n}$ on a vector space $W$. Consider the $4$-tuple $(X, \pi, \nu, U)$, where $\nu(x)=\textrm{dim}(W)$ for all $x\in X$ and  the cocycle $U$ is a constant function on $X$ defined by $U(\tau_k,x)= \rho(\tau_k)$. Then $\phi_{(X, \pi, \nu, U)}$ is a representation of dimension $|X| \textrm{dim}(W)$. In section $4$ we are going to give conditions in the $4$-tuple to obtain an irreducible representation.
\end{exa}

Using this construction of representations, we can see that all the representations of $\B_n$ that satisfy the properties (\ref{condicion (1)}), (\ref{condicion (2)}) and (\ref{condicion (3)}) given in the introduction are of this type. More precisely, they can be parametrized as follows.

From now on for a representation $\psi$ of $\B_n$ fix the notation $\psi_k :=\psi(\tau_k)$. Let $N$ be the algebra of linear operators generated by the spectral projections associated to the operators $\psi_k^2$.

\begin{thm}\label{teorema 2}
Let $\psi:\B_n\rightarrow \textrm{GL}(V)$ be a finite dimensional representation of the braid group $\B_{n}$. Then, $\psi$ satisfies the following conditions, for all $k,j$,
\begin{enumerate}
\item $\psi_k$ is a self-adjoint operator such that $\psi_k^2\neq \lambda 1_V$;
\item $[\psi_k^2, \psi_j^2]=0$;
\item $\psi_k P_{j,l} \psi_k^{-1}\in N$ for all spectral projection $P_{j,l}$ associated to $\psi_j^2$,
\end{enumerate}
if and only if  $\psi$ is equivalent to $\phi_{(X, \pi,\nu, U)}$, for some $4$-tuple  $(X, \pi,\nu, U)$, where the action $\pi$ satisfies that $\pi_k^2=1$ for all $k$ and $U$ is a cocycle such that $U(\tau_k,x)^*=U(\tau_k,\pi_k x)$ for all $x\in X$ and $U(\tau_k,x)$ $U(\tau_k, x)^*$ commute with $U(\tau_j,x)\ U(\tau_j, x)^*$ for all $k,j$.
\end{thm}

\begin{proof}
 As $\psi_k$ is a self-adjoint operator over $V$ for each $k$, $1\leq k\leq n-1$, then $\psi_k^2$ is self-adjoint as well. Let $\psi_k^2=\sum_{l\in I} \lambda_{k,l} P_{k,l}$ be the spectral decomposition of $\psi_k^2$, where the cardinality of $I$ is grater than $1$ because $\psi_k^2\neq \lambda 1_V$.

Note that the index set $I$ does not depend  on $k$. Indeed, each generator of $\B_n$ is conjugate to $\tau_1$ (see
\cite{C} p.655), then for all $k$, $\psi_k$ has the same eigenvalues than $\psi_1$. On the other hand, the eigenvalues of $\psi_k^2$ are the squared of the eigenvalues of $\psi_k$, therefore $\psi_k^2$ have the same eigenvalues than $\psi_1^2$.
Then the cardinality of $I$ is the same for all $k$.

Moreover, the commutativity condition ensures that the operators $\{\psi_k^2\}_{k=1}^{n-1}$ simultaneously diagonalize and $V$ can be decomposed in a direct sum of subspaces $V_{x}=\cap_{k=1}^{n-1} {\rm Im} P_{k,x_k}$, where $x=(x_1, \dots, x_{n-1})$. That is, if
$$X=\{(x_1, \dots, x_{n-1}) : x_i\in I, 1\leq i\leq n-1 {\textrm{  and  }  } \cap_{k=1}^{n-1} {\textrm{Im}} P_{k,x_k}\neq \{0\}\}$$
then $V=\bigoplus_{x\in X} V_x$. Let $\nu:X \rightarrow \N$ be the function defined by $\nu(x):=\textrm{dim}(V_x)$.

Note that there is a correspondence between projections $P_{k,l}$ and the following subset of $X$,
$$
X_{k,l}:=\{x\in X : x_k=l\}
$$
Then $\{x\}=\cap_{j=1}^{n-1} X_{j, x_j}$. With this in mind, the condition (3) says that the projection $\psi_k P_{j,l} \psi_k^{-1}$ corresponds to a subset of $X$ that we will call $X_{\pi_k(j,l)}$. As $\psi_k P_{j,l} \psi_k^{-1}\in N$, the subset $X_{\pi_k(j,l)}$ is union of intersection of the subsets $X_{h,l}$.

Therefore, we can define an action $\pi$ of the group $\B_{n}$ on $X$ given by
\begin{equation}\label{defaccion}
\pi_k(x)=\pi_k(\cap_{j=1}^{n-1} X_{j,x_j}):=\cap_{j=1}^{n-1} X_{\pi_k(j,x_j)}
\end{equation}

We are going to prove that $\pi$ is actually an action in lemma \ref{accion}.

Now, we prove that $\nu(\pi_k x)=\nu(x)$ for all $x\in X$ and all $k$, $1\leq k \leq n-1$. Note that if $v\in V_x=\cap_{j=1}^{n-1} \textrm{Im}P_{j,x_j}$, then $P_{j,x_j}(v)=v$ for all $j$. Therefore, for all $j$, $\psi_k(v)= \psi_k(P_{j,x_j}(v))=P_{\pi_k(j,x_j)} \psi_k(v)$, by condition (\ref{condicion (3)}). This means that $\psi_k(v)\in \cap_{j=1}^{n-1} \textrm{Im}P_{\pi_k(j,x_j)}=V_{\pi_k x}$.

Suppose that there exists $x\in X$ such that $\nu(x)=m$ and $\nu(\pi_k(x))<m$. Let $\{h_1, \dots h_m\}$  be a base of $V_{x}$. Then the set $\{h_1,\dots, h_m\}$ is linearly independent in $V$. As $\psi_k$ is invertible, the set $\{w_1=\psi_k(h_1), \dots, w_m=\psi_k(h_m)\}$ is linearly independent in $V$ as well. But $w_i\in V_{\pi_k x}$, for all $i$. Then $\{w_1, \dots, w_m\}$ is a linearly independent set in $V_{\pi_k x}$. That is a contradiction because $\nu(\pi_k(x))<m$.  Thus, $\nu(\pi_k x)=\nu(x)$ for all $x\in X$.

Now, denote any element $v\in V$ by $v=\sum_{x\in X}x\otimes v_x$ to remark the dependence on $x$ of the component of $v$ in $V_x$.
For each $k$, consider the following operator on $V$,
$$A_k( \sum_ {x\in X} x\otimes v_x)= \sum_{x\in X } x\otimes v_{\pi_k x}$$
Observe that $A_k$ is well defined by the $\pi$-invariance of $\nu$ and it is
invertible, $A_k^{-1}( \sum_{x\in X} x\otimes v_x)=\sum_{x\in X} x \otimes v_{\pi_k^{-1} x}$. Furthermore, they have the following property,

\begin{equation} \label{propiedad (2)}
A_k P_{j,l}=P_{\pi_k^{-1}(j,l)} A_k
\end{equation}

On the other hand, condition (3) says
\begin{equation}\label{propiedad (1)}
 \psi_k P_{j,l}=P_{\pi_k(j,l)} \psi_k
\end{equation}

Let $U(\tau_k):= \psi_k A_k $, it commutes with all the
projection $P_{j,l}$ by (\ref{propiedad (2)}) and (\ref{propiedad (1)}). Then $U(\tau_k)$ decomposes in blocks which acts in each $V_x$, that is there exists $U(\tau_k, x):V_x\rightarrow V_x$ such that

$$ U(\tau_k)(\sum_{x\in X} x\otimes v_x)=\sum_{x\in X} x\otimes U(\tau_k, x) v_x $$

Then $\psi_k= U(\tau_k) A_k^{-1} $, that is
$$
 \psi_k( \sum_{x\in X} x\otimes v_x)= \sum_{x\in X} x\otimes U(\tau_k, \pi_k^{-1} x) v_{\pi_k^{-1} x}
$$

The relations of the braid group imply that the operators $U(\tau_k, x)$ verify the following equations,
$$
\begin{array}{rl}
&U(\tau_k, \pi_k^{-1} x) U(\tau_{k+1}, \pi_{k+1}^{-1} \pi_k^{-1} x) U(\tau_k, \pi_k^{-1}\pi_{k+1}^{-1} \pi_k^{-1} x) = \\
    &\qquad\qquad\qquad\qquad U(\tau_{k+1}, \pi_{k+1}^{-1} x)\  U(\tau_k,\pi_k^{-1}\pi_{k+1}^{-1}x)\ U(\tau_{k+1},\pi_{k+1}^{-1} \pi_k^{-1} \pi_{k+1}^{-1} x)
\end{array}
$$
if $1\leq k\leq n-2$,

$$
U(\tau_k, \pi_k^{-1}x) U(\tau_j,\pi_j^{-1}\pi_{k}^{-1}x)= U(\tau_j, \pi_j^{-1}x) U(\tau_k, \pi_k^{-1}\pi_j^{-1} x)
$$if $|j-k|>1$.

This means that $U:\B_n \times X\rightarrow\textrm{GL}(V)$ is a cocycle. Using the inner product $V$, we obtain that $\psi_k^*(\sum_{x\in X} x\otimes v_x)=\sum_{x\in X} x\otimes U(\tau_k, x)^*v_{\pi_k x}$. Note that under this inner product the subspaces $V_x$ are orthogonal. But $\psi_k$ is self-adjoint, then for all $v=\sum_{x\in X}x\otimes v_x$,
$$
 \sum_{x\in X} x\otimes U(\tau_k, x)^* v_{\pi_k x}= \sum_{x\in X} x\otimes U(\tau_k, \pi_k^{-1} x) v_{\pi_k^{-1} x}
$$

Therefore $U(\tau_k, x)^*v_{\pi_k x}= U(\tau_k, \pi_k^{-1} x)v_{\pi_k^{-1} x}$. But this equality has sense if $\pi_k =\pi_k^{-1}$. In such case $U(\tau_k,x)^*=U(\tau_k,\pi_k x)$ for all $x\in X$.

It rest to see that $U(\tau_k,x)U(\tau_k, x)^*$ commutes with $U(\tau_j,x)U(\tau_j, x)^*$ for all $k,j$. Observe that $\psi_k^2=\psi_k\psi_k^*$, then
\begin{equation}\label{2}
\begin{array}{rl}
\psi_k^2 \psi_j^2&(\sum_{x\in X}x\otimes v_x)=\psi_k\psi_k^* \psi_j\psi_j^*(\sum_{x\in X}x\otimes v_x) \\
&        \\
                                   &=\sum_{x\in X} x\otimes U(\tau_k, \pi_k^{-1}x)U(\tau_k, \pi_k^{-1}x)^*U(\tau_j, \pi_j^{-1} x) U(\tau_j,\pi_j^{-1} x)^* v_{x}
\end{array}
\end{equation}
On the other hand
$$
\psi_j^2 \psi_k^2(\sum_{x\in X}x\otimes v_x)=\sum_{x\in X} x\otimes U(\tau_j, \pi_j^{-1}x)U(\tau_j, \pi_j^{-1}x)^*U(\tau_k, \pi_k^{-1} x) U(\tau_k,\pi_k^{-1} x)^* v_{x}
$$
As $\psi_k^2$ commutes with $\psi_j^2$, then for each $x\in X$,
$$[U(\tau_k, \pi_k^{-1}x)U(\tau_k, \pi_k^{-1}x)^*,U(\tau_j, \pi_j^{-1} x) U(\tau_j,\pi_j^{-1}x)^*]=0$$
\

Conversely, we have already seen that $\phi=\phi_{(X, \pi, \nu, U)}$ define a representation of $\B_n$ on $V=\bigoplus_{x\in X} V_x$. It rests to see that it satisfies the conditions (\ref{condicion (1)}), (\ref{condicion (2)}) and (\ref{condicion (3)}).
\

We have just seen that $\phi_k^*(\sum_{x\in X}x\otimes v_x)=\sum_{x\in X} x\otimes U(\tau_k, x)^*v_{\pi_k x}$
but the hypothesis over $U$ said that $U(\tau_k, x)^*=U(\tau_k, \pi_k x)$, then
$$
\begin{array}{rl}
\phi_k^*(\sum_{x\in X}x\otimes v_x)&=\sum_{x\in X} x\otimes U(\tau_k, x)^*v_{\pi_k x}= \sum_{x\in X} x\otimes U(\tau_k, \pi_k x) v_{\pi_k x}\\
           &        \\
   &= \sum_{x\in X} x\otimes U(\tau_k, \pi_k^{-1} x) v_{\pi_k^{-1} x}=\phi_k(\sum_{x\in X} x\otimes v_x)
\end{array}
$$

Therefore $\phi_k$ is self-adjoint for all $k$. The condition (\ref{condicion (2)}) follows directly from the property that $U(\tau_k, x) U(\tau_k, x)^*$ commutes with $U(\tau_j, x) U(\tau_k, x)^*$ and the calculus (\ref{2}). For the last condition, note that $\phi_k=\sum_{l\in I} \lambda_{k,l} P_{k,l}$, where $P_{k,l}$ is the orthogonal projection over the subspace $\bigoplus_{x\in X_{k,l}} V_x$ with $X_{k,l}=\{x\in X : x_k=l\}$. Then, we have that
$$
\begin{array}{rl}
\phi_k P_{j,l}\phi_k^{-1}(\sum_{x\in X}&x\otimes v_x)=\phi_k P_{j,l}\left( \sum_{x\in X} x\otimes U(\tau_k, x)^{-1}v_{\pi_k x}\right)\\
   &  \\
&= \phi_k\left(\sum_{x\in X} x\otimes \chi_{X_{j,l}}(x)U(\tau_k,  x)^{-1} v_{\pi_k x}\right) \\
  &   \\
&=\sum_{x\in X }x\otimes \chi_{X_{j,l}}(\pi_k^{-1}x)U(\tau_k,\pi_k^{-1} \pi_k x)U(\tau_k, x)^{-1} v_{\pi_k^{-1} \pi_k x} \\
   &   \\
& =\sum_{x\in X}x\otimes \chi_{X_{j,l}}(\pi_k^{-1}x) v_x= P_{\pi_k(X_{j,l})}\left(\sum_{x\in X}x\otimes v_x\right)
\end{array}
$$ (here, $\chi_{X_{j,l}}(x)=1$ if $x\in X_{j,l}$ and $\chi_{X_{j,l}}(x)=0$ in other cases),
hence (3) is true as we wanted.
\end{proof}

This parametrization by $4$-tuples is not a bijective correspondence, that is, different tuples may define equivalent representations. The following lemma complete the proof of the theorem.

\begin{lem}\label{accion}
The map $\pi$ given by (\ref{defaccion}) is an action of $\B_{n}$ on $X$.
\end{lem}
\begin{proof}
We must see that $\cap_{j=1}^{n-1} \pi_k(X_{j,x_j})\neq\O$ and that $\pi_k(x)$ is an element of $X$, that is
$\cap_{j=1}^{n-1} \pi_k(X_{j,x_j})= \cap_{j=1}^{n-1} X_{j,l_j}$ for some $l_j \in I$.

In fact, $x$ is identified with the set $\cap_{j=1}^{n-1} X_{j,x_j}$ which is associated to the non-zero projection
$\prod_{j=1}^{n-1} P_{j,x_j}$. In the same way, $\cap_{j=1}^{n-1} \pi_k(X_{j,x_j})$ is associated to the operator
$$
\prod_{j=1}^{n-1}( \psi(\tau_k) P_{j,x_j} \psi(\tau_k^{-1}))=\psi(\tau_k)( \prod_{j=1}^{n-1} P_{j,x_j}) \psi(\tau_k^{-1})
$$ It is non zero since $\psi(\tau_k)$ is invertible. Hence, $\cap_{j=1}^{n-1} \pi_k(X_{j,x_j})\neq\O $.

Note that for each $j$ and $r$, $1\leq j, r\leq n-1$, there exists $l_r\in I$ such that $X_{r,l_r}\cap \pi_k(X_{j,x_j})\neq \O $. In fact, for each $r$, $\sum_{l\in I} P_{r,l}=1_{V}$. Hence $\psi(\tau_k)P_{j, x_j} \psi(\tau_k^{-1})=\sum_{l\in I} (\psi(\tau_k)P_{j, x_j} \psi(\tau_k^{-1})P_{r,l})$. As the left side is a non-zero projection, there exists $l_r\in I$ such that $\psi(\tau_k)P_{j, x_j} \psi(\tau_k^{-1})P_{r,l}$ is non-zero. Then, $X_{r,l_r}\cap \pi_k(X_{j,x_j})$ is a non empty set.

This says that all the sets $X_{r,l_r}$, $1\leq r\leq n-1$, appear in the expression of $\pi_k(X_{j,x_j})$ as a union of intersections of the subsets $X_{j,l}$ of $X$. That is, if
$$
\pi_k(X_{j,x_j})=\bigcup_{m^j\in L^j} X_{t_1^j,m_1^j}\cap X_{t_2^j,m_2^j}\cap \dots \cap
X_{t_{r^j}^j,m_{r^j}^j}
$$ where $1\leq j \leq n-1$, $m^j= (m^j_1, m^j_2,\dots, m_{r^j}^j)$ and $L^j$ is some index set, then for each $r$ there exists $i$ such that $t_i^j=r$ and $m_i^j=l_r$.

Note that $l_r$ can depend on $j$, but we can choose it such that it does not. In fact, we want to see that for each $r$ there exists $l_r\in I$ such that for all $j$, $1\leq j\leq n-1$, $X_{r,l_r}\cap \pi_k(X_{j,x_j})\neq {\O}$. Suppose that for all $l\in I$, there exists $j'$ such that $X_{r,l}\cap \pi_k(X_{j',x_{j'}})={\O}$. Hence,
$$
{\O}=\cup_{l\in I} (X_{r,l}\cap \pi_k(X_{j',x_{j'}}))=(\cup_{l\in I} X_{r,l})\cap \pi_k(X_{j',x_{j'}})= X\cap \pi_k(X_{j',x_{j'}})
$$ which is a contradiction.

Compute $\cap_{j=1}^{n-1} \pi_k(X_{j,x_j})$,
$$
\begin{array}{ll}
&\bigcap_{j=1}^{n-1} \pi_k(X_{j,x_j})=\bigcap_{j=1}^{n-1} (\bigcup_{m^j\in L^j} X_{t_1^j,m_1^j}\cap
X_{t_2^j,m_2^j}\cap \dots \cap X_{t_{r^j}^j,m_{r^j}^j})\\
                       &=\bigcup_{m^1\in L^1, \dots,m^{n-1}\in L^{n-1}}\left(X_{t_1^1,m_1^1}\cap
X_{t_2^1,m_2^1}\cap \dots \cap X_{t_{r^{1}}^1,m_{r^{1}}^1} \cap \dots\cap X_{t_1^k,m_1^k}\cap \right.\\
                        &\qquad \quad \left.\cap X_{t_2^k,m_2^k}\cap \dots \cap X_{t_{r^k}^k,m_{r^k}^k} \cap\dots \cap
X_{t_1^{n-1},m_1^{n-1}}\cap\dots \cap X_{t_{r^{n-1}}^{n-1},m_{r^{n-1}}^{n-1}}\right)
\end{array}
$$
But if in a term of the union
 \begin{equation} \label{1}
\begin{array}{rl}
&X_{t_1^1,m_1^1}\cap
X_{t_2^1,m_2^1}\cap \dots \cap X_{t_{r^{1}}^1,m_{r^{1}}^1} \cap \dots\cap  X_{t_1^k,m_1^k}\cap X_{t_2^k,m_2^k}\cap\\
                        &\qquad \qquad \cap\dots \cap X_{t_{r^k}^k,m_{r^k}^k} \cap\dots \cap
X_{t_1^{n-1},m_1^{n-1}}\cap\dots \cap X_{t_{r^{n-1}}^{n-1},m_{r^{n-1}}^{n-1}}
\end{array}
\end{equation}
appears the subsets $X_{t_{\alpha}^l,m_{\alpha}^l}$ and $X_{t_{\alpha}^l,m_{\beta}^l}$, with the same first subindex $t_{\alpha}^l$, then the second ones have to be equal, $\alpha=\beta$, or the intersection (\ref{1}) is empty. Therefore, each term of the union is intersection at most of $n-1$ sets, or it is empty. But the $n-1$ sets $X_{r,l_r}$ are the unique sets which appear in the decomposition of $\pi_k(X_{j,x_j})$ for all $j$. Then, they appear in each not empty term of the union.

Therefore $ \bigcap_{j=1}^{n-1} \pi_k(X_{j,x_j})=\cap_{r=1}^{n-1} X_{r,l_r}$ or it is empty. But we have already
seen that it is non empty, so $\pi_k(x)$ is well defined.

On the other hand, the braid group equations imply that
$$
\psi_k \psi_{k+1} \psi_k P_{j,x_j}\psi_k^{-1} \psi_{k+1}^{-1}\psi_k^{-1}= \psi_{k+1}\psi_k \psi_{k+1}P_{j, x_j}
\psi_{k+1}^{-1} \psi_k^{-1} \psi_{k+1}^{-1}
$$
if $1\leq k\leq n-1$ and
$$
\psi_i \psi_k P_{j, x_j} \psi_k^{-1} \psi_i^{-1}= \psi_k \psi_i P_{j, x_j} \psi_i^{-1}
\psi_k^{-1}
$$
if $|i-k|>1$,  hence
$$
\pi_k \pi_{k+1} \pi_k=\pi_{k+1} \pi_k \pi_{k+1}
$$ for all $k$ such that $1\leq k\leq n-2$, and
$$
\pi_k \pi_j =\pi_j \pi_k
$$ if $|j-k|>1$.

Therefore $\pi$ is an action of the braid group $\B_{n}$ on $X$.
\end{proof}

Observe that the proof of the Theorem \ref{teorema 2} can be used for other groups such as Artin groups. This topic will be studied in future works.

 \

One property of self-adjoint representations is that they are completely reducible. In fact, if $W$ is a $\phi$-invariant subspace of $V$, then $W^{\bot}$, the orthogonal complement of $W$, is an invariant subspace of $\phi^*=\phi$. Moreover, if the representations $\phi$ satisfies the hypothesis of the theorem \ref{teorema 2}, the following corollary said that it is a direct sum of irreducible representations of that type. That is
$$
\phi=\bigoplus_{a\in A} \phi_{(X_a,\pi_a, \nu_a, U_a)}
$$

\begin{cor} \label{subespacio}
Let $\left(\phi_{(X, \pi, \nu, U)}, V=\bigoplus_{x\in X} V_x\right)$ be a representation of $\B_{n}$ which satisfies the conditions of the Theorem \ref{teorema 2}. Then, every invariant subspace $W$ of $V$ is a direct sum space $W=\bigoplus_{x\in X_W} W_x$, where
$X_{W}\subset X$, $W_x\subset V_x$ and $W_{\pi_k(x)}=W_x$ for all $x\in X_W$.
\end{cor}
\begin{proof}
Let $W\subseteq V$ be an invariant subspace, then $W$ is also invariant for $\phi(\tau_k)^*=\phi(\tau_k)$. Thus,
$\phi(\tau_k)$ commutes with $P_W$, the orthogonal projection onto the subspace $W$.

Therefore, $\widetilde{\phi_k}:=P_W \phi_k P_W$ is a representation of $\B_{n}$ on $W$ that verifies the hypothesis of the Theorem \ref{teorema 2}. In fact,
$$
\widetilde{\phi_k}^2=P_W \phi_k P_W \phi_k P_W=P_W \phi_k^2 P_W= \sum_{l\in I} \lambda_{k,l}^2 P_W P_{k,l} P_W=\sum_{l\in I} \lambda_{k,l}^2 Q_{k,l}
$$
and
$$
\begin{array}{rl}
\widetilde{\phi_k}^2 \widetilde{\phi_j}^2 &=P_W \phi_k^2 P_W P_W \phi_j^2 P_W =P_W \phi_k^2\phi_j^2 P_W \\
&=P_W \phi_j^2\phi_k^2 P_W= \widetilde{\phi_j}^2\widetilde{\phi_k}^2
\end{array}
$$
Finally
$$
\widetilde{\phi_k} Q_{j,l} \widetilde{\phi_k}^{-1}=P_W \phi_k P_W P_W P_{k,l} P_W P_W \phi_k^{-1} P_W= P_W \phi_k P_{k,l}\phi_k^{-1} P_W
$$ belongs to the algebra generated by the projections $Q_{j,l}$ because $\widetilde{N}=P_W N P_W$.

Then, $\widetilde{\phi}=\phi_{( X_W, \pi_W,\nu_W, P_W U P_W)}$. Moreover
$$
X_W=\{ (x_1, \dots, x_{n-1})\in X :  P_W P_{k, x_k} P_W \neq 0 \textrm{    for all   } k, 1\leq k \leq n-1\}\subseteq X
$$
and
$$
W_x=\cap_k \textrm{ Im} (P_W P_{k,x_k} P_W)=\textrm{ Im} P_W \cap(\cap_k \textrm{ Im}
 P_{k, x_k})
$$
Therefore, $W_x\subset V_x$ for all $x\in X_W$. Finally, the $\pi$-invariance of $\nu_W$ shows that $W_{\pi_k(x_0)}=W_x$ for all $x\in X_W$.
\end{proof}

\section{Known examples: Local Representations}\label{Ejemplos}

Let $V$ be a vector space and $c: V\otimes V\rightarrow V\otimes
V$ an invertible lineal operator, $c$ satisfies the
braid equation if it satisfies the following equality in $V \otimes V
\otimes V= V^{\otimes 3}$
$$
(c\otimes 1_V)(1_V\otimes c)(c\otimes 1_V)=(1_V\otimes c)(c \otimes
1_V)(1_V\otimes c)
$$

In this case $(V,c)$ is called a braid vector space and $c$ a $R$-matrix. On the space $V^{\otimes n}$ one may define a representation of $\B_n$ in
the following way. For each $k=1, \dots , n-1$ let $c(\tau_k):=c_k: V^{\otimes n}\rightarrow V^{\otimes n}$ given by
$c_k=1_{k-1}\otimes c \otimes 1_{n-k-1}$, where $1_k= 1_V\otimes \dots \otimes 1_V$, is the $k$ times tensor
product of the identity on $V$; this representation is called {\it{local}}.

If there exists a basis $\beta=\{v_1, \dots, v_m\}$ of $V$ such that
$c(v_i\otimes v_j)= q_{ij} v_j\otimes v_i$ for all $i,j$, $1\leq i,j\leq m$ where the $q_{ij}$ are
non zero complex numbers, the representation is called of {\it{diagonal}} type \cite{AG}. In this case, the operators $c_k$ satisfy the hypothesis of the theorem \ref{teorema 2} if $q_{ij}=\overline{q_{ji}}$ for all $i,j$, $1\leq i,j \leq m$. In fact, $c_k$ is self-adjoint because $q_{ij}=\overline{q_{ji}}$, moreover

$$ c_k^2 (v_{j_1}\otimes \dots \otimes v_{j_n})=q_{j_k,
j_{k+1}}^2 v_{j_1}\otimes\dots \otimes v_{j_k}\otimes v_{j_{k+1}}
\otimes \dots \otimes v_{j_n}
$$
If $S=\left\{(a,b): a,b\in \{1, \dots, \textrm{dim}(V)\} \right\}$, we have that
$$
c_k^2 =\sum_{(a,b)\in S} q_{(a,b)}^2 P_{k,(a,b)}
$$
where $P_{k, (a,b)}$ is the projection over the subspace of
$V^{\otimes n}$ generated by all the vectors $v_{j_1}\otimes \dots
\otimes v_{j_n}$ such that $j_k=a$ y $j_{k+1}=b$.

As $c_k^2$ is diagonal for all $k$, $c_k^2$ and $c_j^2$ commute for all $j$.

We have to see the condition (3) of theorem \ref{teorema 2}. In fact,
$$
c_kP_{j, (a, b)}c_k^{-1}=\left\{
\begin{array}{ll}
              P_{j, (a,b)}        &  \textrm{    if    } |j-k|>1 \\
              P_{j, (b, a)}       & \textrm{    if    }  j=k       \\
              \sum_{c=1}^{\dim V} P_{k-1, (a, c)}\wedge P_{k, (c, b)}       & \textrm{    if    }  j=k-1  \\
              \sum_{c=1}^{\dim V} P_{k, (a, c)}\wedge P_{k+1, (c, b)}       & \textrm{    if    }  j=k+1
\end{array}\right.
$$ where $P\wedge Q$ is the orthogonal projection over the subspace $\textrm{Im}P\cap \textrm{Im}Q$.

Therefore, following the proof of Theorem \ref{teorema 2}, the representation is equivalent to $\phi_{(X, \pi, \nu,  U)}$, where
\begin{enumerate}
\item $X=\{(x_1, \dots , x_{n-1}) : x_i=(a_i, b_i) , a_i,b_i \in \{1, \dots,
 \dim V\} \textrm{  and  } b_i=a_{i+1}\}$,
\item The action $\pi$ is defined by
$$
\begin{array}{rl}
\pi_k&((a_1, a_2),(a_2, a_3), \dots, (a_{n-1}, b_{n-1}))=\\
&=((a_1, a_2),\dots, (a_{k-1}, a_{k+1}), (a_{k+1}, a_k),(a_k, a_{k+2}), \dots, (a_{n-1}, b_{n-1}))
\end{array}
$$
and $\pi_k^{-1}(x)=\pi_k(x)$.
\item  $\nu(x)=1$ for all $x\in X$,
\item  $U(\tau_k, x)=q_{x_k}$.
\end{enumerate}

Let $\alpha :V^{\otimes n} \rightarrow V$ be the lineal operator
defined in the basis of $V^{\otimes n}$ by
$$
\alpha(v_{j_1}\otimes \dots \otimes v_{j_n})=\chi_{((j_1, j_2),
(j_2,j_3), \dots ,(j_{n-1}, j_n))}
$$

It verifies that $\alpha(c_k(v_{j_1}\otimes \dots \otimes v_{j_n}))=
\phi_k (\alpha(v_{j_1}\otimes \dots \otimes v_{j_n}))$ showing the equivalence of the representations.

\section{Irreducibility and Equivalence of Representations}

Given the $4$-tuple $(X, \pi, \nu, U)$, we will analyze the irreducibility of the associated representation $\phi_{(X, \pi, \nu, U)}$.

\begin{thm}\label{irreduc}
Let $\phi_{(X, \pi, 1, U)}$ be a self-adjoint representation of $\B_n$ that satisfies the hypothesis of the theorem \ref{teorema 2}.
If the action $\pi$ is transitive, then the representation is irreducible.
\end{thm}

\begin{proof}
Observe that if $U(\tau_k,x)=c$ for all $x\in X$ and for all $k$, then $\phi=\phi_{(X, \pi, 1, U)}$ does not satisfy the hypothesis of the theorem \ref{teorema 2}. Indeed, in such cases
$$\phi_k^2(\sum x\otimes v_x)= \sum x\otimes c^2v_v=c^2\sum x\otimes v_x$$
therefore $\phi_k^2=c^2 1_v$.

Let $W\subseteq V$ be an invariant subspace, then by corollary \ref{subespacio}, $\widetilde{\phi}:=P_W \phi P_W =
\phi_{( X_W, \pi_W, \nu_W, P_W U P_W)}$, where  $X_W \subseteq X$, $W_x\subset V_x$ and $W_{\pi_k x}=W_x$.

As $\nu(x)=1$ for all $x\in X$,  $W_x\subseteq \C$, then $W_x=0$ or $W_x=\C$, for $x\in X_W$. Let $Y:=\{x\in X: W_x=0\}$. Now, as $X=\pi(\B_n)x_0$ and $W_{\pi_k x}=W_x$ for all $x$, then $Y=\emptyset$ or $Y=X$.  Therefore  $W=0$ or $W=V$.
\end{proof}

Note that if $U(\tau_k,x)=c$ for all $x\in X$ and for all $k$, then we have already seen that the representation $\phi=\phi_{(X, \pi, 1, U)}$ does not satisfy the hypothesis of the theorem \ref{teorema 2}. Even then we can consider the representation $\phi$ of $\B_n$. Note that it is not irreducible. In fact $v=\sum_{x\in X} x\otimes 1$ generates an invariant subspace of $V=\bigoplus_{x\in X}\C \cong \C^{|X|}$.

\
We can obtain other irreducible representations of $\B_n$.
 \begin{thm}\label{induced}
Let $\rho$ be an irreducible self-adjoint representations of
$\B_n$ on a  vector space $W$ of dimension $m$ such that $1<m<\infty$. Let
$(X, \pi, \nu, U)$ be a $4$-tuple where $\B_n$ acts transitively on $X$, $\nu(x)=m$ and $U(\tau_k,
x):=\rho (\tau_k)$ for all $x\in X$. Then $\phi_{(X, \pi, \nu, U)}$ is an irreducible representation
of $\B_n$ of dimension $|X|\textrm{dim}(W)$.
\end{thm}

\begin{proof}
$\phi_{(X, \pi, \nu, U)}$ is a representation that satisfies the conditions of the Theorem \ref{teorema 2}. Let $S$ be an invariant subspace of
$V:=\bigoplus_{x\in X}W$. By corollary  \ref{subespacio}, $S=\bigoplus_{x\in X_S} S_x$, where $X_S\subset X$, $S_x$ is a subspace
of $W$ and $S_{\pi_k x}=S_x$. Given $y\in X_S$, we want to see that $S_y$ is an invariant subspace by
$\rho$. Let $v\in S_y$ be a non zero vector, let us see that for
all $k$, $1\leq k\leq n-1$, $\rho(\tau_k)v \in S_y$. For each
$k$, let $v_k= \pi_k y \otimes v \in S$.
Then $ \phi_k (v_k)= \pi_k y\otimes U(\tau_k, \pi_k^{-1} \pi_k y) v \in S$ and $U(\tau_k, y) v\in S_{\pi_k y}=S_y$ for all $y$ by corollary \ref{subespacio}. But $ U(\tau_k, y) v= \rho(\tau_k)v$. Therefore $\rho(\tau_k)v \in S_y$ as we wanted. Thus $S_y$ is $\rho$-invariant for all $y\in X_S$, and $S_y=0$ or $W$ because $\rho$ is irreducible. Now,  $\pi$ acts transitively on $X$ and $S_{\pi_k y}=S_y$ then $S_y=0$ for all $y$ or $S_y=W$ for all $y$. Therefore $S=0$ or $S=V$.
\end{proof}

Note that this proposition permit us to construct, from a fixed irreducible representation, a sequence of irreducible representation of $\B_n$ of strictly increasing dimension. Therefore we can give explicit examples of irreducible representations of dimension  arbitrary large.

\begin{thm}\label{dimgrande}
For any positive integer $M$, there exists an irreducible representation of $\B_n$ of dimension greater that $M$.
\end{thm}

\begin{proof}
Let $\phi_m$ the representation given in the introduction. Let $(X_1, \pi_1, \nu_1, U_1)$ be a $4$-tuple where $X_1$ is the set of $n$-tuples with $m$ ones and $n-m$ zeros, $\pi_1$ is the action given by permutation of the coordinates of $X$, $\nu_1$ is constant equal to $\textrm{dim}(V_m)=\left( \begin{smallmatrix} n\\\noalign{\medskip}m \end{smallmatrix}\right)$ and $U(\tau_k, x)=\phi_m(\tau_k)$. Then, by the previous theorem $\phi_1=\phi_{(X_1, \pi_1, \nu_1, U_1)}$ is an irreducible representation of $\B_n$ of dimension
$$|X_1|\left( \begin{smallmatrix} n\\\noalign{\medskip}m \end{smallmatrix}\right)=\left( \begin{smallmatrix} n\\\noalign{\medskip}m \end{smallmatrix}\right)^2$$

If $\left( \begin{smallmatrix} n\\\noalign{\medskip}m \end{smallmatrix}\right)^2>M$, we have obtained the wished representation.
If not, we repeat the process. Let $(X_2, \pi_2, \nu_2, U_2)$ such that $X_2:=X_1$, $\pi_2:=\pi_1$, $\nu_2(x)=\left( \begin{smallmatrix} n\\\noalign{\medskip}m \end{smallmatrix}\right)^2$ and $U_2(\tau_k,x)=\phi_1(\tau_k)$. Thus, the dimension of $\phi_2$ is $\left( \begin{smallmatrix} n\\\noalign{\medskip}m \end{smallmatrix}\right)^3$. Repeating this process, there exists $r\geq 1$ such that $\left( \begin{smallmatrix} n\\\noalign{\medskip}m \end{smallmatrix}\right)^r>M$. Thus, $\phi_{r-1}$ is the wanted representation.
\end{proof}

It is known that transitive actions of a group $G$ on a set $X$ are in correspondence with subgroups $H$ that fix an element $x_0$ of $X$. We wonder if the construction of the proposition \ref{induced} is the Induced Representation of the representation $\rho$ restricted to $H$. In general, the answer is no. Before showing an example where the constructions are different, recall that the support of the character of an induced representation is concentrated in $H$.

Let $m>4$ and let $\rho=\phi_m$ be the irreducible representation of $\B_n$ given in the introduction. Let $X=X_m$ be the set of $n$-tuples with $m$ ones and $n-m$ zeros. The action $\pi$ of permuting the coordinates of $X$ is transitive. Let $\phi=\phi_{(X, \pi, \nu, U)}$ the representation constructed in theorem \ref{induced} with the previous parameters. Finally, let $\chi$ be the character of this representation. If $H$ is the subgroup of $\B_n$ that fixes $x_0=(1, \dots, 1, 0, \dots, 0)$, then $\chi(\tau_m)>\frac{(m-2)!}{(n-2)!(m-n)!}>1\neq 0$ but $\tau_m$ does not belong to $H$. Therefore $\phi$ is not an induced representation.

\

We are going to finish this work with a results about equivalence of representations associated to $4$-tuples. Suppose that $\phi_{(X_1, \pi_1, \nu_1, U_1)}$ is equivalent to $\widetilde{\phi}_{(X_2, \pi_2, \nu_2, U_2)}$, where $X_1=X_2$, $\pi_1=\pi_2$ and the $4$-tuples satisfy the hypothesis of the Theorem \ref{teorema 2}. Then, there exist an operator $c:V=\bigoplus V_x\rightarrow W= \bigoplus W_x$ such that, for each $k$, $c \phi_k=\widetilde{\phi}_k c$. Therefore, the spectral projections of $\widetilde{\phi}_k$ are conjugated to the spectral projections of $\phi_k$. Thus, $\textrm{dim}(V_x)=\textrm{dim}(\bigcap_{k=1}^{n-1}\textrm{Im}P_{k,x_k})=\textrm{dim}(W_x)$ and $\nu_1(x)=\nu_2(x)$ for all $x\in X$, that is $V_x\cong W_x$, suppose $V_x=W_x$ for all $x$. Therefore $c$ is decomposed in blocks $c(x)$ that acts in each $V_x$ and they satisfy that
\begin{equation}\label{equiv}
c(\pi_k x) U(\tau_k,x)c(x)^{-1}=\widetilde{U}(\tau_k,x)
\end{equation}

Therefore a necessary condition for two $4$-tuples $(X, \pi, \nu_1, U_1)$ and $(X, \pi, \nu_2, U_2)$ to define equivalent representations is the existence of an operator $c:V\rightarrow W$ such that it decomposes in blocks $c(x)$ that acts in each $V_x$ such that (\ref{equiv}) is true.

\end{document}